\documentclass[12pt]{article}
\usepackage{amssymb}
\usepackage{amsmath}
\usepackage{amsthm}
\usepackage{tikz}
\usepackage{fullpage}
\usepackage{mathrsfs}
\usepackage{verbatim}
\usepackage{enumitem}
\usepackage{bbm}

\usepackage[left=1in, right=1in, top=1in, bottom=1in, includefoot]{geometry}
\usepackage{graphicx}
\usepackage{subcaption}
\usepackage{bbm}
\usepackage{multirow}
\usepackage{hyperref}

\numberwithin{equation}{section}

\newtheorem{theorem}{Theorem}[section]
\newtheorem{lemma}[theorem]{Lemma}
\newtheorem{definition}[theorem]{Definition}

\newtheorem{corollary}[theorem]{Corollary}

\newtheorem{question}[theorem]{Question}

\title{Cyclic Cellular Automata and Greenberg--Hastings Models on Regular Trees}
\author{Jason Bello and David Sivakoff}
\date{}
\begin{document}
\maketitle
\begin{abstract}

We study the cyclic cellular automaton (CCA) and the Greenberg--Hastings model (GHM) with $\kappa\ge 3$ colors and contact threshold $\theta\ge 2$ on the infinite $(d+1)$-regular tree, $T_d$. When the initial state has the uniform product distribution, we show that these dynamical systems exhibit at least two distinct phases. For sufficiently large $d$, we show that if $\kappa(\theta-1) \le d - O(\sqrt{d\kappa \ln(d)})$, then every vertex almost surely changes its color infinitely often, while if $\kappa\theta \ge d + O(\kappa\sqrt{d\ln(d)})$, then every vertex almost surely changes its color only finitely many times. Roughly, this implies that as $d\to \infty$, there is a phase transition where $\kappa\theta/d = 1$. For the GHM dynamics, in the scenario where every vertex changes color finitely many times, we moreover give an exponential tail bound for the distribution of the time of the last color change at a given vertex.

\end{abstract}

\section{Introduction}
The cyclic cellular automaton (CCA) and the Greenberg-Hastings model (GHM) are discrete models for excitable dynamical systems, which have been studied extensively on the integer lattices $\mathbb{Z}^d$, where they are known to exhibit a wide range of interesting behavior. Following the recent work of~\cite{gravner2018}, we study the CCA $(\xi_t)_{t\in\mathbb{Z}_{\ge 0}}$ and GHM $(\gamma_t)_{t\in\mathbb{Z}_{\ge 0}}$ dynamics on the infinite $(d+1)$-regular tree, $T_d$, with $\kappa\ge 3$ colors and contact threshold $\theta\ge 2$.

Formally, given an initial assignment of colors to the vertices of a graph $G = (V,E)$, $\xi_0\in\{0,1,\ldots,\kappa-1\}^{V}$, the CCA dynamics are determined by the rule 
\begin{align*}
\xi_{t+1}(v) =
\begin{cases}
 \xi_{t}(v)+1 \mod \kappa & \text{if } |\{u\sim v: \xi_t(u) =\xi_t(v)+1\mod \kappa\}| \geq \theta,\\
 \xi_{t}(v) &\text{otherwise,}
 \end{cases}
\end{align*}
for all $v\in V$, where $u\sim v$ means that $(u,v)\in E$, and $|S|$ denotes cardinality of a set $S$. In other words, if a vertex has at least $\theta$ neighbors that have color exactly one greater than it on the color wheel, then it is ``painted" by the larger color at the next time step. Another common metaphor used in CCA is the predator-prey relationship. Every color is an animal, and every animal has exactly one predator and exactly one prey. In this way, a prey will be ``eaten" (and consequently, replaced) by its predator if there are at least $\theta$ adjacent predators. This model was proposed by Bramson and Griffeath as the deterministic counterpart of the cyclic particle system~\cite{BandG89}.

The GHM dynamics were proposed by Greenberg and Hastings as a model for excitable media \cite{Greenberg1978}, and we accordingly refer to the sets of colors $\{0\}$, $\{1\}$, and $\{2, \dots, \kappa-1\}$ as resting, excited, and refractory states. Given an initial configuration $\gamma_0\in\{0,1,\ldots,\kappa-1\}^{V}$, the GHM dynamics are determined by the rule
$$
\gamma_t(v) =  
\begin{cases} 
\gamma_t(v) +1 \mod \kappa,  & \text{if } \gamma_t(v)\geq 1 \text{ or } |\{u\sim v: \gamma_t(u) =1\}| \geq \theta, \\
0 & \text{otherwise}
\end{cases}
$$
for all $v\in V$. That is, vertices at rest become excited if they have at least $\theta$ excited neighbors, otherwise they remain at rest. Once excited, a vertex enters a refractory period, which lasts $\kappa-2$ steps before the vertex returns to the resting state. In this paper we only considered initial configurations drawn from the uniform product measure, but some work has been done for $\kappa=3$ and $\theta=1$ on $\mathbb Z^d$ for $d\geq 2$ with regards to more general translation invariant product measures \cite{Durrett:1991ql}.

These seemingly simple local update rules lead to a range of complex phenomena, including spiral nucleation, periodicity, traveling waves and spatial clustering. In the present study, we focus on two distinct phenomena. We say that the CCA dynamics $(\xi_t)_{t\in \mathbb{Z}_{\ge0}}$ \textit{fixates} (on $G$) iff there exists a configuration $\xi_\infty \in \{0,1,\ldots, \kappa-1\}^{V}$ such that $\lim_{t\to \infty}  \xi_t(v) = \xi_\infty(v)$ for all  $v\in V$. Likewise, we say that the GHM dynamics $(\gamma_t)_{t\in \mathbb{Z}_{\ge0}}$ \textit{fixates} iff $\lim_{t\to \infty}  \gamma_t(v) = 0$ for all $v\in V$. Note that fixation does not imply that the limiting configuration is reached in finite time, but it does imply that each vertex changes color only finitely many times. We say that the vertex $v\in V$ \textit{fluctuates} for either the CCA or GHM dynamics iff it changes its color infinitely many times, and we say that the (CCA or GHM) dynamics \textit{fluctuates} iff every vertex fluctuates.

We assume that our graph $G$ is the infinite, undirected $(d+1)$-regular tree $T_d$, and we abuse notation by referring to both the graph and its vertex set by  $T_d$. We assume that $\xi_0$ and $\gamma_0$ are chosen randomly, and have the uniform product distribution on $\{0,\ldots, \kappa-1\}^{T_d}$, which we denote by $\mathbb P = \mathbb P_{\kappa,d}$. We use the term \textit{almost surely} to mean $\mathbb{P}$-almost surely. Our main results are the following two theorems.

\begin{theorem}\label{thm:fixation}
Assume $\kappa\ge 3$ and $\theta\ge 2$ and $d\ge 3$. If $\kappa\left(\theta-3\sqrt{d\ln(d)} \right) \geq d$, then CCA and GHM fixate on $T_d$ almost surely. Moreover, for the GHM dynamics, $(\gamma_t)$, if $\tau$ is the last time the root vertex is in the excited state, then for $n \ge 0$  
$$\mathbb P (\tau > n) \leq \left(\frac 1d\right)^{n}.$$
%
\end{theorem}

\begin{theorem}\label{thm:fluctuation}
Assume $\kappa\ge 3$ and $\theta\ge 2$ and $d\ge 4$. If $\kappa (\theta-1)  \leq d - \sqrt{6d\kappa\ln(d)}$, then the CCA and GHM dynamics fluctuate on $T_d$ almost surely, and $\displaystyle \lim_{n\to\infty} \xi_{n\kappa}$ and $\displaystyle\lim_{n\to\infty} \gamma_{n\kappa}$ exist almost surely.
\end{theorem}

For Theorem~\ref{thm:fixation}, $d\ge 3$ is sufficiently large for the estimates in our proofs to hold, but when $d=3$, for instance, we need $\theta\ge 6$ for the left side of the condition to be non-negative. When $\theta>d$, fixation is trivial to prove for all $\kappa\ge 3$, so the condition in Theorem~\ref{thm:fixation} is only nontrivial when there exist values of $\theta$ for which $3\sqrt{d\ln d}< \theta \le d$, which requires $d\ge 31$ (and also $\theta\ge 31$).  For smaller values of $\theta$ and $d$, we prove the following sufficient conditions for fixation.

\begin{theorem}\label{thm:CCA small theta}
If $d\ge\theta\ge 3$ and $\kappa\ge 3$ and $\kappa(\theta-2)\ge 9e d^{1 + \frac1{\theta-2}}$, then the CCA dynamics, $(\xi_t)$, fixates on $T_d$ almost surely. If $d\ge \theta = 2$ and $\kappa\ge 12 d^3$, then the CCA dynamics fixates on $T_d$ almost surely.
\end{theorem}
It is clear from our proof that the condition in Theorem~\ref{thm:CCA small theta} also guarantees fixation for GHM, but for GHM dynamics we can use a simple first-moment argument to give the following improved sufficient condition for fixation.

\begin{theorem}\label{thm:GHM small theta}
Assume $d\ge \theta\ge 2$ and $\kappa\ge 3$. If $\kappa\ge e\left(\frac{de}{\theta}\right)^{1+ 1/(\theta^\kappa - 1)}$, then the GHM dynamics, $(\gamma_t)$, fixates on $T_d$ almost surely. Moreover,  if $\tau$ is the last time the root vertex is in the excited state, for $n\ge \kappa$ we have
$$\mathbb P (\tau > n) \leq 2{d+1\choose \theta} \exp\left[-\theta^{n-\kappa+1}\right].$$
Also, for $\theta=2$ and $\theta=3$ and some small values of $d$, the GHM dynamics fixates on $T_d$ almost surely for all $\kappa$ greater than or equal to the values given in the table below.
\begin{equation*}
\begin{array}{|c||c|c|c|c|c|c|c|c|}
\hline d& 2 & 3 & 4 & 5 & 6 & 7 & 8 & 9 \\ \hline\hline
\theta=2 & 3 & 5 & 7& 8 & 10 &11 &12 & 14 \\ \hline
\theta=3 & 3 & 3 & 3& 4 & 5 & 5& 6 & 7 \\ \hline
\end{array}
\end{equation*}
For $\theta=d\ge 2$, the GHM dynamics fixates almost surely for all $\kappa\ge 3$.
\end{theorem}

For the condition in Theorem~\ref{thm:fluctuation} to be satisfied, since $\kappa\ge 3$ and $\theta\ge 2$, we at least need $d\ge 3+\sqrt{18 d\ln d}$, which requires $d\ge 87$. We do not know whether fluctuation occurs for much smaller values of $d$, such as those in the table above. The estimates used to prove Theorem~\ref{thm:fluctuation} have not been optimized, but we do not think this will lead to a significant improvement.
%


The study of CCA and GHM on infinite trees was initiated recently by Gravner, Lyu and the second author~\cite{gravner2018}, who showed fluctuation and computed the fluctuation rate when $\theta=1$ and $\kappa=3$. Previously, these models had been studied extensively on the integer lattices $\mathbb Z^d$. For $d=1$ and $\theta=1$, Fisch found that the CCA almost surely fixates for every $\kappa\geq 5$, and almost surely fluctuates for $\kappa \leq 4$ \cite{Fisch:1990gl}. This result confirmed that the CCA exhibits the same phase transition as was proved earlier by Bramson and Griffeath \cite{BandG89} for the cyclic particle systems on $\mathbb Z$, wherein vertices update their colors asynchronously. Fisch~\cite{Fisch:1992} later proved that the 3-color CCA ``clusters'' on $\mathbb Z$, meaning that any fixed pair of sites will have the same color with probability tending to 1 as $t\to\infty$. In fact, using random walk coupling arguments, he proved the stronger result that the probability of two sites having different colors at time $t$ is $1-\Theta(t^{-1/2})$. Durrett and Steif~ \cite{Durrett:1991ql} proved the analogous result for the 3-color GHM on $\mathbb{Z}$, showing that the probability of a site having color $1$ at time $t$ decays like $\Theta(t^{-1/2})$. Fisch and Gravner~\cite{Fisch1995} later extended this result to the $\kappa$-color GHM on $\mathbb Z$ for $\kappa\ge 4$. Lyu and the second author~\cite{LYU20191132} generalized the random walk coupling argument to accommodate spatially corollated initial colorings, which allowed them to recover these results for the 3-color CCA and GHM, and prove an analogous result for the 3-color ``firefly'' cellular automaton model that was introduced by Lyu~\cite{Lyu201528}. Recently, Foxall and Lyu~\cite{Foxall:2018pi} showed that the $3$- and $4$-color cyclic particle systems cluster on $\mathbb{Z}$, but highlight computing the clustering rate as an open problem.

On $\mathbb Z^2$ with $\theta=1$, Fisch, Gravner and Griffeath proved that the CCA~\cite{Fisch1991} and GHM~\cite{Fisch1993} fluctuate for all $\kappa \ge 3$ by analyzing the emergence of stable period objects (SPOs), which are local configurations that exhibit temporally periodic behavior regardless of the colors of vertices outside of the region. The existence of SPOs plays an important role in both CCA and GHM on $\mathbb Z^d$ with $d\ge 2$. Indeed, for neighborhoods of large radius $\rho$ in $\mathbb Z^2$, by proving existence of SPOs and using percolation arguments, Fisch, Gravner and Griffeath~\cite{Fisch:1991km,Fisch1993}, Durrett~\cite{Durrett:1992} and Durrett and Griffeath~\cite{DG:1993} prove that the CCA and GHM on $\mathbb Z^2$ exhibit distinct qualitative phases depending on the asymptotic value of $\frac{\kappa \theta}{\rho^2}$. In particular, they show that if $\frac{\kappa \theta}{\rho^2}$ is sufficiently small, then SPOs exist, which drive fluctuation and local $\kappa$-periodicity, and if $\frac{\kappa \theta}{\rho^2}$ is sufficiently large, then the dynamics fixate in bounded time. In two dimensions, a vertex with $\ell^p$-ball neighborhood of radius $\rho$ has degree $\sim c_p \rho^2$, so our results on $T_d$ are consistent with these long-range two-dimensional findings. In fact, as a  consequence of Theorems~\ref{thm:fixation} and~\ref{thm:fluctuation}, we obtain a sharp transition in the following sense. 

\begin{corollary}\label{cor:asymp}
Assume that as $d\to\infty$ either $\theta / \sqrt{d\ln d} \to \infty$ or $\kappa \sqrt{\frac{\ln d}{d}} \to 0$. Then the following hold.
\begin{enumerate}[label=\alph*)]
\item If $\displaystyle \liminf_{d\to\infty} \frac{\kappa\theta}{d} > 1$, then for all large enough $d$, CCA and GHM fixate on $T_d$ almost surely.
\item If $\displaystyle \limsup_{d\to\infty} \frac{\kappa\theta}{d} < 1$, then for all large enough $d$, CCA and GHM fluctuate on $T_d$ almost surely.
\end{enumerate}
\end{corollary}
\begin{proof}
Fixation is assured by Theorem~\ref{thm:fixation} for large $d$ if $\frac{\kappa\theta}{d}-3\kappa\sqrt{\frac{\ln d}{d}} \geq 1$. This is clearly satisfied if $\liminf_{d} \frac{\kappa\theta}{d} > 1$ and $\kappa \sqrt{\frac{\ln d}{d}} \to 0$. Rearranging, fixation is assured for large $d$ if 
$$
\frac{\kappa\theta}{d}\left(1 - 3\frac{\sqrt{d\ln d}}{\theta} \right) \ge 1,
$$ 
which is satisfied if $\liminf_{d} \frac{\kappa\theta}{d} > 1$ and $\frac{\sqrt{d\ln d}}{\theta}\to 0$. The verifications for fluctuation are similar.
\end{proof}



On trees, finite SPOs for CCA and GHM appear to be topologically prohibited~\cite{gravner2018}. Instead, to prove Theorem~\ref{thm:fluctuation} we identify infinite SPOs by adapting percolation methods of Balogh, Peres and Pete~\cite{balogh_peres_pete_2006} who study bootstrap percolation on $T_d$. A key idea of theirs that we use is the definition of a $k$-fort.
\begin{definition}\label{k-forts}
For a graph $G=(V,E)$ and a set of vertices $H\subseteq V$, let $\text{deg}_H(v) = |\{ w\in H: (w,v) \in E\}|$ be the number of neighbors of $v$ that lie in $H$. Let $k\ge 0$. We call $S\subseteq V$ a \textbf{$k$-fort} iff $S$ is connected and each $v\in S$ has $\text{deg}_{V\setminus S}(v) \leq k$. 
\end{definition}
For bootstrap percolation with occupation threshold $\theta$, a vacant $(\theta-1)$-fort is stable, and existence or nonexistence of these structures determines whether full occupation can occur. Unlike bootstrap percolation, however, the CCA and GHM are non-monotone cellular automata, and there is no inherent monotonicity in $d$ or $\kappa$. Nonetheless, the analog of a SPO for CCA or GHM on $T_d$ is a ``rainbow'' colored $\theta$-ary subtree (see Definition~\ref{rainbow def}), and we can use a more quantitative adaptation of the methods from~\cite{balogh_peres_pete_2006} to show that percolation of these structures leads to fluctuation. To prove Theorem~\ref{thm:fixation}, we rely on similar methods, with some careful modifications, to identify $(\theta-2)$-forts of stable vertices, which cause the rest of $T_d$ to fixate.

\section{Proof of Fixation}
We designate a root vertex $\rho \in T_d$.  For a pair of neighboring vertices in $T_d$, we call the one that is closer to $\rho$ the parent and the one that is farther from $\rho$ the child, and for any vertex in $T_d$ we refer to its neighbors farther from $\rho$ as its children, and its neighbor closer to $\rho$ as its parent. Throughout our proofs, we will assume addition and subtraction of colors ($\xi$ and $\gamma$) are taken modulo $\kappa$. 

We begin with the definition of a rigid set, which we state for CCA $(\xi_t)$, though it applies also for GHM $(\gamma_t)$.

\begin{definition}\label{rigid}
Let $T$ be a rooted tree and $\xi_0\in \{0,\ldots, \kappa-1\}^{T}$. We say a set of vertices $S\subseteq T$ is \textbf{rigid} (for $\xi_0$) iff for any pair of neighboring vertices $p,v\in S$ such that $p$ is the parent of $v$, we have $\xi_0(v) - \xi_0(p) \neq 1$.
Furthermore, we say a non-root vertex is {rigid} if the set containing it and its parent is rigid.
\end{definition}

Rigid $k$-forts play an important roll in fixation due to the following observation.

\begin{lemma} \label{lem:rigid k-fort}{\ }
\begin{enumerate}
\item[i)] For CCA: If $S\subseteq T_d$ is a rigid $(\theta-2)$-fort for $\xi_0$, then $(\xi_t(v))_{t\ge0}$ is a constant sequence for every $v\in S$.
\item[ii)] For GHM: If $S\subseteq T_d$ is a rigid $(\theta-2)$-fort for $\gamma_0$, then $\gamma_t(v) \neq 1$ for all $t\ge 1$ and $v\in S$, and so $\gamma_t(v)=0$ for all $t\ge \kappa-1$ and $v\in S$.
\end{enumerate}
\end{lemma}
\begin{proof}
We first show  $(\xi_t(v))_{t\ge0}$ is constant for every $v\in S$. For a contradiction, suppose $t\ge 1$ is the smallest such that there exists $v\in S$ with $\xi_t(v)\neq \xi_{t-1}(v)$. Then $v$ must have at least $\theta$ neighbors $w$ such that $\xi_{t-1}(w) - \xi_{t-1}(v) = 1$. Since $v$ has at most $(\theta-2)$ neighbors outside of $S$, and at most one of its neighbors within $S$ can be its parent, it must have at least one child $w\in S$ such that $\xi_0(w) -\xi_0(v) = \xi_{t-1}(w) - \xi_{t-1}(v) = 1$. This contradicts $S$ being rigid in $\xi_0$.

We now show that $\gamma_t(v)\neq 1$ for all $t\ge 1$ and all $v\in S$. We prove this by induction on $t$. Let $v\in S$. For $\gamma_1(v)=1$ to occur, we must have $\gamma_0(v)=0$ and $v$ must have at least $\theta$ neighbors $w$ such that $\gamma_0(w) =1$. Since $v$ has at most $(\theta-2)$ neighbors outside of $S$, and at most one of its neighbors within $S$ can be its parent, it must have at least one child $w\in S$ such that $\gamma_0(w)= \gamma_0(w)- \gamma_0(v) = 1$, which contradicts $S$ being rigid in $\gamma_0$. Therefore, $\gamma_1(v) \neq 1$ for all $v\in S$. Now suppose $\gamma_t(v) \neq 1$ for all $v\in S$. Then every vertex of $S$ has at most $(\theta-2)$ neighbors outside of $S$, and so has at most $(\theta-2)$ neighbors that can be in state $1$ at time $t$. This implies $\gamma_{t+1}(v) \neq 1$ for all $v\in S$, which completes the induction. It follows that $\gamma_t(v) = 0$ for all $t\ge \kappa-1$ and $v\in S$ since non-zeros initially in $S$ will reset to $0$ in at most $\kappa-1$ steps, and $0$s must stay $0$ since no new $1$s ever emerge in $S$.
\end{proof}



Now the strategy of our proof is as follows. First, we show that an infinite rooted $d$-ary tree contains a rigid $(\theta-2)$-fort that includes the root with large probability. Second, we use this to show that all infinite non-backtracking paths from the root of $T_d$ will intersect a rigid $(\theta-2)$-fort almost surely, resulting in a finite, potentially non-fixated, connected subtree containing the root. Last, we show that $(\xi_t)$ fixates over any such subtree. The first two steps follow arguments similar to what was done in \cite{balogh_peres_pete_2006}, though we make more explicit estimates to obtain quantitative results. Also, it is crucial for our asymptotic results that we identify rigid $(\theta-2)$-forts, rather than ``strongly'' rigid $(\theta-1)$-forts (see Definition~\ref{strongly rigid}), as suggested by the proof of~\cite{balogh_peres_pete_2006}, otherwise Corollary~\ref{cor:asymp} part (a) would require $\liminf d/(\kappa\theta)>2$, and we would not see the sharp transition. 

We will use a well known Chernoff bound, which we include here for reference.
\begin{lemma}\label{chernoff bound}
Let $X$ be a binomially distributed random variable with mean $\mu$. Then for any $\delta\in [0,1]$,
$$\mathbb P(X \leq (1-\delta)\mu) \leq e^{-\frac{\delta^2 \mu}{2}}.$$
\end{lemma}
We obtain the following binomial probability bounds, which will be applied in our proof of fixation.

\begin{lemma}\label{cdf bounds}
For $d\ge 3$,
  if $\kappa\left(\theta-3\sqrt{d\ln(d)}\right) \geq d$, then
\begin{equation}\label{main event}\mathbb P \left(\textnormal{Binom}\left(d, (1-d^{-2})\left(1-1/\kappa\right)\right) \leq d-\theta+1\right) \leq \frac{1}{d^2}\end{equation}
and
\begin{equation}\label{side event}\mathbb P\left(\textnormal{Binom}\left(d-1, (1-d^{-2})\left(1-1/\kappa\right)\right) \leq d-\theta+2 \right)\leq \frac{1}{d^2}.\end{equation}
\end{lemma}

\begin{proof}
Assume first that $d\ge 3$ and $\kappa\left(\theta-3\sqrt{d\ln(d)}\right) \geq d$. We will only prove (\ref{main event}) since the proof of (\ref{side event}) is nearly identical. Since $\kappa(\theta-3) > \kappa\left(\theta-3\sqrt{d\ln(d)}\right) \geq d$, we have 
$$\delta := \frac{d(1-\frac{1}{d^2})(1-\frac1\kappa) - (d-\theta+1)}{d(1-\frac{1}{d^2})(1-\frac1\kappa)}> 0.$$
Applying Lemma \ref{chernoff bound} with this value of $\delta$ gives
\begin{align*}
&d^{2} \cdot \mathbb P(\text{Binom}\left(d, (1-d^{-2})\left(1-1/\kappa\right)\right) \leq d-\theta+1)  \\
 &\hspace{3cm} \leq \exp\left[ -\frac{d((1-d^{-2})(1-1/\kappa)- (1-\theta/d+ 1/d))^2}{2(1-d^{-2})(1-1/\kappa)} + 2\ln(d)  \right] \\
&\hspace{3cm} \leq \exp\left[ -\frac{d}{2}((1-d^{-2})(1-1/\kappa)- (1-\theta/d+1/d))^2 + 2\ln(d)\right]\\
&\hspace{3cm} =\exp\left[ -\frac{d}{2}\left(\frac{\theta}{d} -\frac{1}{d^2} - \frac1\kappa -\frac{1}{d} + \frac{1}{\kappa d^2} \right)^2 + 2\ln(d)\right]\\
&\hspace{3cm} \leq \exp\left[ -\frac{d}{2}\left(\frac{\theta}{d} - \frac1\kappa-\frac{1}{d^2}  -\frac{1}{d} \right)^2 + 2\ln(d)\right].\\
\intertext{By our assumption, we have $\frac{\theta}{d} -\frac{1}{\kappa} \geq 3\sqrt{\frac{\ln(d)}{d}}$ and for $d\ge 3$ we have $\frac{1}{d^2} + \frac{1}{d} \leq \sqrt{\frac{\ln(d)}{d}} $. Hence,}
&d^{2} \cdot \mathbb P(\text{Binom}\left(d, (1-d^{-2})\left(1-1/\kappa\right)\right) \leq d-\theta+1)  \\
&\hspace{3cm} \leq \exp\left[ -\frac{d}{2}\left( 2\sqrt{\frac{\ln(d)}{d}} \right)^2 + 2\ln(d)\right]\\
&\hspace{3cm} = 1.\\
\end{align*}
Thus for $d\ge 3$, we have (\ref{main event}) as desired, and (\ref{side event}) follows in a similar manner, where $d\ge 3$ is again sufficient.
\end{proof}

With these tools in hand, we move on to the first step in our proof. We define the rooted $d$-ary tree $\mathcal{T}_d$ to be the infinite tree with root node $\rho$ such that every vertex has degree $d+1$ except for the root, which has degree $d$. That is, every vertex of $\mathcal{T}_d$ has $d$ children. We let $\bar \xi$ be a uniform random coloring of the vertices in $\mathcal{T}_d$.
\begin{lemma}\label{rigid subtree}
Let  $E$ be the event that $\bar\xi$ does not contain a rigid $(\theta -2)$-fort that includes the root of $\mathcal{T}_d$. For $d\ge 3$, if  $\kappa\left(\theta-3\sqrt{d\ln(d)}\right) \geq d$, then $\mathbb P( E) \leq \frac{1}{d^2}$.
\end{lemma}
\begin{proof}
Given the color of the root $\rho\in \mathcal{T}_d$, we may think of the rigid subtree containing the root as a Galton-Watson tree with offspring distribution Binom$(d, 1-\frac 1\kappa)$. The success probability in the offspring distribution is $1-\frac 1\kappa$ because for the subtree being constructed to remain rigid, each child must only avoid the color that is one greater than the parent's color. The root $\rho$ is in a rigid $(\theta-2)$-fort if and only if $\rho$ has at least $(d-\theta+2)$ rigid children and, in the forest obtained by removing $\rho$ from $\mathcal{T}_d$, at least $(d-\theta+2)$ of these must be contained in a rigid $(\theta-2)$-fort. The maximal disjoint trees obtained by removing $\rho$ from $\mathcal{T}_d$ are copies of $\mathcal{T}_d$, and since $E^c$ is independent of the color of the root in $\mathcal{T}_d$, we have that each child of $\rho$ is independently both rigid and in a rigid $(\theta-2)$-fort of $\mathcal{T}_d\setminus\{\rho\}$ with probability $(1-\frac 1\kappa)(1-\mathbb P(E))$.  It follows that $\mathbb P(E)$ is a fixed point of the function
$$B_1(x) := \mathbb P\left(\text{Binom}\left(d, (1-x)\left(1-1/\kappa\right)\right) \leq d-\theta+1\right),$$
for $x\in[0,1]$ (see Figure \ref{fig:B_1(x) plots}). The value $x=1$ is always a fixed point of $B_1$, but we will show that $\mathbb P(E)$ is always the smallest fixed point of $B_1(x)$ in $[0,1]$. First, using calculus and binomial coefficient identities we have
$$\frac{\partial}{\partial x} B_1(x) = d(1-  1/\kappa) \mathbb P\left(\text{Binom}\left(d-1, (1-x)\left(1-1/\kappa\right)\right) = d-\theta+1\right).$$
So for all $x\in[0,1)$, $B_1(x)$ is a monotonically increasing function with $B_1(0) > 0$. Let $y_n$ be the probability that $\rho$ is not in a rigid $(\theta-2)$-fort in $\mathcal{T}_d$ truncated at distance $n$, then $y_0 = 0$, $y_{n+1} = B_1(y_{n})$, and $y_n \to \mathbb P(E)$. By Kleene's fixed-point theorem, $\mathbb P(E)$ is the smallest fixed point of $B_1(x)$. Furthermore, if there exists $x_0\in (0,1)$ such that $B_1(x_0)\leq  x_0$, then $\mathbb P(E) \leq x_0$. By \eqref{main event} in Lemma \ref{cdf bounds}, we have $B_1(d^{-2})\leq  d^{-2}$, so $\mathbb P(E) \leq d^{-2}$ as desired.
\end{proof}

\begin{figure}[h!]
  \centering
  \begin{subfigure}[b]{0.45\linewidth}
    \includegraphics[width=\linewidth]{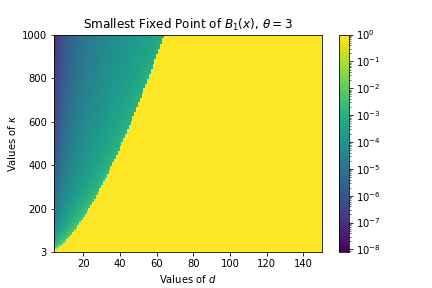}
  \end{subfigure}
  \begin{subfigure}[b]{0.45\linewidth}
    \includegraphics[width=\linewidth]{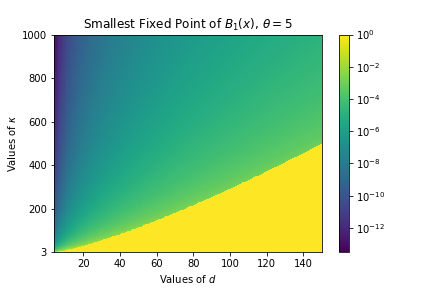}
  \end{subfigure}
  \caption{The values of the smallest fixed point of $B_1$ as $\kappa$ and $d$ vary for $\theta=3$ (left) and $\theta=5$ (right). The smallest fixed point of $B_1$ determines the probability that a rigid $(\theta-2)$-fort does not contain the root of a $d$-ary tree. This probability is equal to $1$ on the yellow regions, and is strictly less than $1$ on the green-blue regions. For fixed $d$ and $\theta$, once $\kappa$ exceeds the threshold value where the smallest fixed point drops below $1$, the smallest fixed point of $B_1$ quickly approaches 0. The threshold for $\kappa$ appears to be nonlinear in $d$, at least for these small values of $\theta$ and $d$.}
  \label{fig:B_1(x) plots}
\end{figure}

\begin{lemma}\label{path rigid}
Suppose $\kappa\left(\theta-3\sqrt{d\ln(d)}\right) \geq d \ge 3$. Then with probability 1 every infinite non-backtracking path in $T_d$ starting at $\rho$ will intersect a rigid $(\theta-2)$-fort. Moreover, if $E_n$ is the event that there exists a non-backtracking path of length $n$ starting at $\rho$ that does not intersect a rigid $(\theta-2)$-fort, then
$$
\mathbb{P}(E_n)\le d^{-n}.
$$
\end{lemma}
\begin{proof}
By Lemma \ref{rigid subtree}, if $E$ is the event that an infinite rooted $d$-ary tree does not contain a rigid $(\theta-2)$-fort beginning at the root in the initial configuration, then $\mathbb P (E) \leq \frac{1}{d^2}$ for sufficiently large $d$. 
Now let $E'$ be the event that in an infinite rooted tree with $d-1$ children at the root and $d$ children everywhere else, there does not exist a rigid subtree $S$ that contains the root, $(d-\theta+3)$ children of the root and $(d-\theta+2)$ children of every other vertex in $S$. Note that such a rigid subtree exists if and only if the root of this tree has at least $(d-\theta+3)$ rigid children and, in the forest obtained by removing the root, at least $(d-\theta+3)$ of these must be contained in rigid $(\theta-2)$-forts. Therefore, we have
\begin{align*}
\mathbb P (E') &=\mathbb P\left(\text{Binom}\left(d-1, (1-\mathbb P(E))\left(1-1/\kappa\right)\right) \leq d-\theta+2\right)\\
&\leq \mathbb P\left(\text{Binom}\left(d-1, (1-d^{-2})\left(1-1/\kappa\right)\right) \leq d-\theta+2 \right)\leq \frac{1}{d^2}.
\end{align*}
The first inequality follows from Lemma \ref{rigid subtree} and the fact that the Binomial cdf is decreasing in its success probability. The second inequality follows from \eqref{side event} in Lemma \ref{cdf bounds}.

Let $(v_\ell)_0^{n}$ be a non-backtracking path in $T_d$ of length $n$ such that $v_0=\rho$, and define the events 
$$
F_k := \{v_k \text{ is in a rigid $(\theta-2)$-fort disjoint from } \{v_\ell: \ell\neq k \}\}
$$
 for $k= 0,\dots, n$. For fixed $n$, note that $F_0, \ldots, F_n$ are independent. Additionally, $\mathbb P(F_k) \geq1- \mathbb P(E')$ for $k=1, \ldots, n-1$ and $\mathbb P(F_0) = P(F_n) \geq 1 - \mathbb P(E)$. So we have
\begin{align*}
\mathbb P((v_k)_0^n &\text{ does not intersect a rigid $(\theta-2)$-fort})\\
&\leq \mathbb P \left(\bigcap_{k=0}^n F_k^c\right) =\prod_{k=0}^n \left[1 - \mathbb P(F_k) \right]\leq \mathbb P(E)^2 P (E')^{n-1} \leq  \left(\frac{1}{d^{2}}\right)^{n+1}.
\end{align*}
Let $E_n$ be the event that at least one such path of length $n$ does not intersect a rigid $(\theta-2)$-fort. Then
$$\mathbb P(E_n) \leq (d+1)d^n \cdot\left(\frac{1}{d^{2}}\right)^{n+1} \leq \left(\frac{1}{d}\right)^n$$
and so 
$$ \sum_{n=0}^\infty \mathbb P(E_n)  \leq \sum_{n=0}^\infty\left(\frac{1}{d}\right)^n< \infty.$$
Thus by the Borel-Cantelli lemma, all infinite non-backtracking paths will intersect a rigid $(\theta-2)$-fort almost surely. 
\end{proof}

The last ingredient before the proofs that CCA and GHM fixate is that $(\xi_t)$ and $(\gamma_t)$ fixate on finite trees.

\begin{lemma}\label{fix lemma} Suppose $\theta\ge 2$ and let $T$ be a finite, rooted tree.
\begin{enumerate}
\item[i)]  If $\deg(v) \le \kappa (\theta-1)$ for all $v\in T$, then $(\xi_t)$  fixates on $T$ for all initial colorings $\xi_0 \in \{0,\ldots, \kappa-1\}^T$.
\item[ii)] If $T$ has depth $n$, then for every initially coloring $\gamma_0\in \{0,\ldots, \kappa-1\}^T$ and every $v\in T$ we have $\gamma_t(v) \neq 1$ for all $t\ge n+1$, and $\gamma_t(v)=0$ for all $t\ge n+\kappa$.
\end{enumerate}
\end{lemma}
\begin{proof}
For CCA $(\xi_t)$, we prove (i) by induction on the depth of the tree. If the tree has depth $n=0$, and consists of just a single node, then clearly $(\xi_t)$ fixates at time $0$. Suppose now that $(\xi_t)$ fixates for all $\xi_0$ on every tree with depth $n-1$ satisfying the degree condition, and let $T$ be a tree with depth $n$ satisfying the degree condition. Observe that it suffices to show that all vertices in $T$ at depth at least $n-1$ fixate by some step $N$. Indeed, let $T'$ be the depth-$(n-1)$ rooted subtree of $T$ with only the depth-$n$ vertices of $T$ removed, and let $(\xi'_t)$ be the CCA dynamics on $T'$ with $\xi'_0(v) = \xi_N(v)$ for $v\in T'$. Then since the vertices at depth $n-1$ in $T$ do not change colors after time $N$, we have $ \xi'_{t}(v) = \xi_{t+N}(v)$ for all $t\ge 0$ and $v\in T'$. Since $(\xi'_t)$ fixates on $T'$ by the induction hypothesis, we have that $(\xi_t)$ fixates on $T$. It remains to show that vertices at depth at least $n-1$ fixate.

Let $v\in T$ be a vertex at depth $n-1$ from the root. Since $\theta\ge 2$, all leaves of $T$ never change color, and this includes all vertices at depth $n$ from the root, and in particular all children of $v$. Suppose for the sake of contradiction that $v$ fluctuates. For each color $c\in \{0, \ldots, \kappa-1\}$, let $A(c) = \{ u : u \text{ is a child of $v$ and }\xi_0(u) = c\}$. If $n=1$, for fluctuation to occur we must have $|A(c)|\ge \theta$ for every $c$, and $\deg(v) = \sum_c |A(c)| \ge \kappa\theta > \kappa(\theta-1)$, which gives a contradiction. If $n\ge 2$, for $v$ to fluctuate we must have $|A(c)| \geq \theta-1$ for every $c$, and $ \deg(v)-1 = \sum_{c=0}^{\kappa-1} |A(c)| \geq  \kappa (\theta-1),$ so $\deg(v)>\kappa(\theta-1)$. This again gives a contradiction, so $v$ must fixate in $(\xi_t)$.

For GHM, to prove (ii) we claim that if $T$ has depth $n$, then at time $t \in [0, n+1]$, no vertices at depth at least $n-t+1$ are excited (color $1$) at time $t$. It follows that at time $n+1$, there are no excited vertices in $T$, so $\gamma_t\equiv 0$ for all $t\ge n+\kappa$. We prove the claim by induction on $t$. Clearly there are no vertices at depth at least $n+1$, so the claim holds for $t=0$. Suppose there are no excited vertices at depth at least $n-t+1$ at time $t\in [0,n]$, and consider a vertex $v\in T$ at depth at least $n-t$. Then $v$ has at most one excited neighbor at time $t$, and therefore cannot be excited at time $t+1$. This completes the induction and the proof.
\end{proof}

We can now complete the proof of Theorem~\ref{thm:fixation}.  
\begin{proof}[Proof of Theorem~\ref{thm:fixation}]
We start by proving fixation for CCA. It suffices to show that $\rho$ fixates almost surely. Let $A = \{v\in T_d: v \text{ is in a rigid $(\theta -2)$-fort} \}$ and let $C_\rho$ be the connected component containing $\rho$ in $T_d\setminus A$. Lemma \ref{path rigid} implies $C_\rho$ is finite almost surely. Let $T$ be the induced subtree on the vertices $C_\rho\cup \partial C_\rho$, where $\partial C_\rho\subset A$ is the set of vertices outside $C_\rho$ with one neighbor in $C_\rho$, and we take $\partial C_\rho=\{\rho\}$ if $C_\rho=\emptyset$. By Lemma~\ref{lem:rigid k-fort}, the vertices in $\partial C_\rho$ never change color, which implies the dynamics on $T$ (with the initial coloring inherited from $T_d$, but ignoring vertices outside $T$) are the same as the dynamics on $T_d$ within $C_\rho\cup\partial C_\rho$. Since $\kappa\left(\theta-3\sqrt{d\ln(d)}\right)\geq d$, we have $\kappa(\theta-1) \ge d$, and Lemma \ref{fix lemma} implies that $(\xi_t)$ fixates on $T$ which contains $\rho$. This completes the proof for CCA.

Now we prove fixation and the exponential tail bound for GHM. Let $\tau=\sup\{t\ge 0 : \gamma_t(\rho)=1\}$, and let $A$, $C_\rho$, $\partial C_\rho$ and $T$ be as above. Let $E_n$ be the event that there exists a non-backtracking path of length $n$ starting at $\rho$ that does not intersect a rigid $(\theta-2)$-fort, so that Lemma~\ref{path rigid} implies $\mathbb{P}(E_n)\le d^{-n}$. Observe that $T$ has depth at least $n+1$ if and only if $E_n$ occurs. By Lemma~\ref{lem:rigid k-fort}, we have $\gamma_t(v)\neq 1$ for all $t\ge 1$ and $v\in \partial C_\rho$, and this property is maintained even if all vertices of $T_d$ outside $C_\rho\cup\partial C_\rho$ are initially removed. Therefore, the dynamics in $T$ and $T_d$ are the same on $C_\rho\cup\partial C_\rho$. Lemma~\ref{fix lemma} implies that for $n\ge 0$,
\begin{align*}
\mathbb{P}(\tau\ge n+1) \le \mathbb{P}(T \text{ has depth at least } n+1) = \mathbb{P}(E_{n}) \le d^{-n},
\end{align*}
which completes the proof.
\end{proof}

\section{Proof of Fluctuation}\label{proof of fluc}
We prove fluctuation on $T_d$ using similar percolation arguments, except now ``rainbow'' subtrees take the place of rigid $k$-forts. We state the definition for CCA, but it applies in the same way to GHM.

\begin{definition}\label{rainbow def}
Let $T$ be an infinite tree and $\xi_0\in \{0,\ldots, \kappa-1\}^{T}$. We call a rooted subtree $R\subseteq T$ a \textbf{rainbow subtree} of $T$ (for $\xi_0$) iff for each $v \in R$ and each child $w\in R$ of $v$ we have $\xi_0(w) = \xi_0(v) +1$.
\end{definition}
The following simple observation justifies the definition.
\begin{lemma}\label{lem:rainbow flux}
If $R$ is an infinite $\theta$-ary rainbow subtree of $T_d$ for $\xi_0$ (resp. $\gamma_0$), then $\xi_{t+1}(v) = \xi_t(v)+1$ (resp. $\gamma_{t+1}(v) = \gamma_t(v)+1$) for all $t\ge 0$ and $v\in R$.
\end{lemma}

We will show that an infinite $d$-ary tree has an infinite $\theta$-ary rainbow subtree that contains the root with large probability, then use this to show that almost surely every infinite non-backtracking path in $T_d$ intersects an infinite $\theta$-ary rainbow subtree. We will then show that finite trees with leaves that increment their colors at every step must fluctuate and eventually be $\kappa$-periodic. Finally, we will put these pieces together to complete the proof of fluctuation.

\begin{lemma}\label{rainbow subtree}
Let $d\ge 2$, $\theta\ge 2$ and $\kappa\ge 3$. Suppose $E$ is the event that an infinite rooted $d$-ary tree does not contain an infinite $\theta$-ary rainbow subtree that contains the root in the initial coloring. If $\kappa (\theta-1) \leq d - \sqrt{6d\kappa\ln(d)}$, then $\mathbb P(E) \le \frac{1}{d^2}$.
\end{lemma}
\begin{proof}

Let $E$ be the event that an infinite rooted $d$-ary tree $\mathcal{T}_d$ does not contain a rainbow subtree in the initial coloring, whose root coincides with the root of $\mathcal{T}_d$. Note that $E$ is independent of the color of the root. Let $\rho$ be the root of this $d$-ary tree. Given the color of $\rho$, we have a $\theta$-ary rainbow subtree  rooted at $\rho$ if and only if at least $\theta$ children of $\rho$ have color $\xi_0(\rho) +1$ and each is the root of a (disjoint) rainbow subtree. For each child of $\rho$, this occurs independently with probability $(1-\mathbb P(E))(\frac 1\kappa)$. It follows that $\mathbb P(E)$ is a fixed point of the function
$$B_2(x) = \mathbb P(\text{Binom}(d, (1-x)/\kappa) \leq \theta-1)$$
on $[0,1]$ (see Figure \ref{fig:B_2(x) plots}).  $B_2$ always has $x=1$ as a fixed point, and we will show that $\mathbb P(E)$ is always the smallest fixed point of $B_2$. Using calculus and binomial coefficient identities we see that
$$\frac{\partial}{\partial x} B_2(x) = \frac d\kappa \cdot \mathbb P\left(\text{Binom}\left(d-1, (1-x)/\kappa\right) = \theta-1\right).$$
So $B_2(x)$ is an increasing function on $[0,1]$, with $B_2(0) > 0$. Let $y_n$ be the probability that there is no finite (perfect) $\theta$-ary rainbow subtree of the finitie tree obtained by truncating $\mathcal{T}_d$ at level $n$. Then $y_0 = 0$, $y_{n+1} = B_2(y_{n})$, and $y_n \to \mathbb P(E)$ as $n\to\infty$. By Kleene's fixed-point theorem, $\mathbb P(E)$ is the smallest fixed point of $B_2$. Furthermore, if there exists $x_0\in (0,1)$ such that $B_2(x_0)\leq  x_0$, then $\mathbb P(E) \leq x_0$. 
To bound $B_2(x)$ we use Lemma \ref{chernoff bound} with 
$$\delta = \frac{d(1-x)- \kappa(\theta-1)}{d(1-x)}.$$
We have $\delta \geq 0$ whenever 
\begin{equation}\label{chernoff cond flux} x \leq \frac{d - \kappa(\theta-1)}{d}.\end{equation}
Since $\kappa(\theta-1) \leq d - \sqrt{6d\kappa \ln(d)}$, $x$ satisfies (\ref{chernoff cond flux}) when $x \leq \sqrt{\frac{6\kappa\ln(d)}{d}}$. Let $x_0 := d^{-2}$ and note that $d^{-2} \leq \sqrt{\frac{6\kappa\ln(d)}{d}}$ for all $d\ge 2$ and $\kappa\ge 3$, so

\begin{equation}\label{B_2 estimate}
\begin{aligned}
x_0^{-1} \cdot B_2(x_0)  &=x_0^{-1} \cdot  \mathbb P(\text{Binom}(d, (1-x_0)/\kappa) \leq \theta-1) \\
&\leq \exp\left[ -\frac{\left[ \frac{d}{\kappa}(1-x_0) - (\theta-1)\right]^2}{2\frac{d}{\kappa}(1-x_0)} -\ln(x_0) \right]\\
&= \exp\left[ -\frac{ \left[ d(1-x_0) - \kappa(\theta-1)\right]^2}{2d\kappa(1-x_0)} -\ln(x_0) \right]\\
&\le \exp\left[ -\frac{ \left[ d- \kappa(\theta-1)-dx_0 \right]^2}{2d\kappa} -\ln(x_0) \right]\\
&\leq \exp\left[ -\frac{ \left[ \sqrt{6d\kappa\ln(d)}-d^{-1} \right]^2}{2d\kappa} +2\ln(d) \right]\\
&= \exp\left[ -3\ln(d) \left[1 - 1/ \sqrt{6d^3\kappa\ln(d)}\right]^2  +2\ln(d) \right]\\
&< 1
\end{aligned}
\end{equation}
for $d\ge 2$ and $\kappa\ge 3$. Thus, $\mathbb P(E) \le \frac{1}{d^2}$ as desired.
\end{proof}

\begin{figure}[h!]
  \centering
  \begin{subfigure}[b]{0.45\linewidth}
    \includegraphics[width=\linewidth]{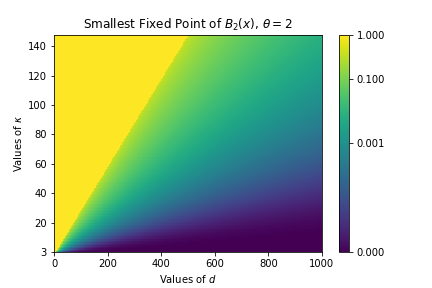}
  \end{subfigure}
  \begin{subfigure}[b]{0.45\linewidth}
    \includegraphics[width=\linewidth]{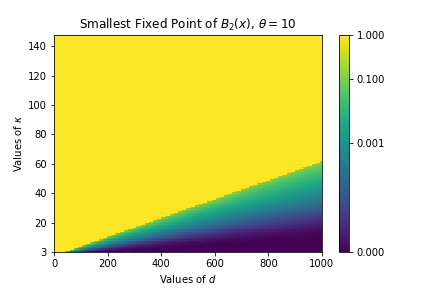}
  \end{subfigure}
  \caption{The value of the smallest fixed point of $B_2$ for $\theta=2$ (left) and $\theta=10$ (right) as a function of $\kappa$ and $d$. This fixed point is equal to the probability that an infinite $\theta$-ary rainbow subtree does not appear at the root of an infinite $d$-ary tree. The threshold value of $\kappa$, below which this probability is smaller than 1, appears to be linear in $d$ for small $\theta$.}
  \label{fig:B_2(x) plots}
\end{figure}

\begin{lemma}\label{rainbow path}
If $\theta\ge 2$, $\kappa\ge 3$, $d\ge 4$ and $\kappa (\theta-1)  \leq d - \sqrt{6d\kappa \ln(d)}$, then almost surely every infinite non-backtracking path from the root in $T_d$ will intersect an infinite $\theta$-ary rainbow subtree. 
\end{lemma}
\begin{proof}
By Lemma \ref{rainbow subtree}, if $E$ is the event that an infinite rooted $d$-ary tree does not contain a rainbow subtree beginning at the root in the initial configuration, then $\mathbb P (E) \leq \frac{1}{d^2}$ for $\kappa, d$ and $\theta$ satisfying the conditions of the lemma. Now let $E'$ be the event that an infinite rooted tree with $d-1$ children at the root and $d$ children of every other vertex does not contain an infinite $\theta$-ary rainbow subtree that shares the same root in the initial coloring. Then,
\begin{align*}
\mathbb P (E') &= \mathbb P( \text{Binom}(d-1, (1-\mathbb P(E))/\kappa) \leq \theta-1)\\
&\leq  \mathbb P( \text{Binom}(d-1, (1-(d-1)^{-2})/\kappa) \leq \theta-1)\leq \frac{1}{(d-1)^2} \leq \frac{2}{d^2}
\end{align*}
for $d\geq4$. In the first inequality we use Lemma \ref{rainbow subtree} and the fact that the Binomial cdf is decreasing in its success probability. The second inequality follows from~\eqref{B_2 estimate} after replacing $d$ by $d-1$.

Let $(v_k)_0^{n}$ be a non-backtracking path in $T_d$ of length $n$ such that $v_0=\rho$, and let 
$$
F_k = \left\{v_k \text{ is the root of an infinite $\theta$-ary rainbow subtree disjoint from $\{v_\ell : \ell\neq k\}$}\right\}
$$
 for $k= 0,\dots, n$. Note that for a given path, $F_0, \ldots, F_n$ are independent. Additionally, $\mathbb P(F_k) \geq1- \mathbb P(E')$ for all $k\in [1,n-1]$ and $\mathbb P(F_k) \geq 1 - \mathbb P(E)$ for $k=0,n$. Therefore,
\begin{align*}
&\mathbb P((v_k)_0^n \text{ does not intersect an infinite $\theta$-ary rainbow subtree})\\
&\hspace{2cm} \le \mathbb P \left(\bigcap_{k=0}^n F_k^c\right) 
=\prod_{k=0}^n 1 - \mathbb P(F_k) 
\leq \mathbb P (E)^2 \mathbb P (E')^{n-1} \leq \left(\frac{2}{d^{2}}\right)^{n+1}.
\end{align*}
Let $E_n$ be the event that there exists a non-backtracking path of length $n$ from $\rho$ that does not intersect an infinite $\theta$-ary rainbow subtree, then
$$\mathbb P(E_n) \leq (d+1)d^n \cdot\left(\frac{2}{d^{2}}\right)^{n+1} \leq \left(\frac{2}{d}\right)^{n}$$
and so 
$$ \sum_{n=0}^\infty \mathbb P(E_n)  \leq \sum_{n=0}^\infty \left(\frac{2}{d}\right)^{n}< \infty.$$
Thus by the Borel-Cantelli lemma, all infinite non-backtracking paths from $\rho$ will intersect an infinite $\theta$-ary rainbow subtree almost surely. 
\end{proof}

We now prove Theorem~\ref{thm:fluctuation}.

\begin{proof}[Proof of Theorem~\ref{thm:fluctuation}]
We give the proof for CCA, $(\xi_t)$, and the proof for GHM, $(\gamma_t)$, is identical. Let 
$$
A = \{v \in T_d: \exists N \in \mathbb{Z}_{\ge 0}, \; \xi_{t+1}(v) = \xi_t(v)+1 \;\; \forall t \geq N\},
$$
be the set of vertices that eventually change their state at every time-step. To prove Theorem~\ref{thm:fluctuation}, it is equivalent to show $A=T_d$ almost surely, for which it suffices to show that $\rho\in A$ almost surely. Let $C_\rho$ be the connected component containing the root $\rho$ in $T_d\setminus A$, where $C_\rho=\emptyset$ iff $\rho\in A$. Let $R\subseteq T_d$ be the set of all vertices that are contained in infinite $\theta$-ary rainbow subtrees. Lemma~\ref{lem:rainbow flux} implies $R\subseteq A$. By Lemma~\ref{rainbow path}, the connected component containing $\rho$ in $T_d\setminus R$ is finite almost surely, so $C_\rho\subseteq T_d\setminus A \subseteq T_d\setminus R$ is also finite almost surely. 

Suppose for the sake of contradiction that $C_\rho \neq \emptyset$. Let $v$ be a leaf in $C_\rho$, which exists since $C_\rho$ is finite. Since all $d$ children of $v$ in $T_d$ are in $A$, there exists $N \ge 0$ such that the children of $v$ all change colors at every time-step after time $N$. Since $\kappa < d/(\theta-1)$, by the pigeonhole principle, at each time $t$ there exists a set of at least $\theta$ children of $v$ that all share the same color. Call one such set of vertices $S_t$. For each $t\geq N$, since the children of $v$ increment their colors at every step we can choose $S_t = S_N =: S$. Since $v\in C_\rho\subseteq T_d\setminus A$, the collection of times at which $v$ does not change colors is infinite, so after time $N$ the color of the vertices in $S$ will eventually catch up to the color of $v$. That is, there exists $N' \geq N$ such that $\xi_{N'}(c) = \xi_{N'}(v) +1$ for all $c\in S$. Since $|S|\ge \theta$, we have $\xi_{N'+1}(v) = \xi_{N'}(v)+1$, and since $N'\ge N$, we have $\xi_{N'+1}(c) = \xi_{N'}(c)+1 =  \xi_{N'}(v) +2$. By induction, we have $\xi_{t+1}(v) = \xi_{t}(v)+1$ for all $t\ge N'$, so $v \in A$, but this contradicts $v\in C_\rho$. Therefore, $C_\rho = \emptyset$ almost surely.
\end{proof}

\section{Fixation for smaller $\theta$}

\subsection*{Fixation for CCA}
We begin with the definition of a strongly rigid set, which replaces the role of a rigid set for smaller $\theta$.

\begin{definition}\label{strongly rigid}
Let $T$ be a rooted tree and $\xi_0\in \{0,\ldots, \kappa-1\}^{T}$. We say a set of vertices $S\subseteq T$ is \textbf{strongly rigid} iff for any pair of neighboring vertices $u,v\in S$ we have $\xi_0(v) - \xi_0(u) \neq 1$.
Furthermore, we say a non-root vertex is strongly {rigid} if the set containing it and its parent is strongly rigid.
\end{definition}

For large values of $\theta$, we seek rigid $(\theta-2)$-forts to guarantee fixation, but for smaller values of $\theta$ and larger $\kappa$, it is more likely that we will find strongly rigid $(\theta-1)$-forts in the initial configuration, which also guarantee fixation, as stated in the next lemma. The proof of Lemma~\ref{lem:strongly rigid k-fort} is nearly the same as the proof of Lemma~\ref{lem:rigid k-fort}, so we omit it.

\begin{lemma} \label{lem:strongly rigid k-fort}{\ }
\begin{enumerate}
\item[i)] For CCA: If $S\subseteq T_d$ is a strongly rigid $(\theta-1)$-fort for $\xi_0$, then $(\xi_t(v))_{t\ge0}$ is a constant sequence for every $v\in S$.
\item[ii)] For GHM: If $S\subseteq T_d$ is a strongly rigid $(\theta-1)$-fort for $\gamma_0$, then $\gamma_t(v) \neq 1$ for all $t\ge 1$ and $v\in S$, and so $\gamma_t(v)=0$ for all $t\ge \kappa-1$ and $v\in S$.
\end{enumerate}
\end{lemma}

Instead of a Chernoff bound, for smaller $\theta$ we use the following binomial bound.

\begin{lemma}\label{binomial bound}
Let $X$ be a binomially distributed random variable with mean $\mu$. Then for any $k\ge 1$,
$$
\mathbb P(X\ge k) \le \left(\frac{\mu e}{k}\right)^k.
$$
\end{lemma}
\begin{proof}
Suppose $X$ counts the number of successes in $n$ independent trials with success probability $p$, so $\mu=np$. The expected number of size-$k$ subsets of the $n$ trials such that all $k$ trials are successes is ${n\choose k} p^k$, and Stirling's approximation gives $k!\ge (k/e)^k$, so Markov's inequality gives
$$
P(X\ge k) \le {n\choose k} p^k \le \frac{n^k}{k!} p^k \le \frac{(np)^k}{(k/e)^k} = \left(\frac{\mu e}{k}\right)^k.
$$
\end{proof}

\begin{lemma}\label{binomial small theta}
If $d\ge\theta\ge 3$ and $p = \frac{1}{3e} d^{-(\theta-1)/(\theta-2)}$ and $\kappa(\theta-2)\ge 9e d^{1 + \frac1{\theta-2}}$, then 
\begin{equation} \label{binomial small 1}
P\left(\textnormal{Binom}\left(d, (1-p)\left(1-2\kappa^{-1}\right)\right) \leq d-\theta \right) \le p,
\end{equation}
and
\begin{equation}\label{binomial small 2}\mathbb P \left(\textnormal{Binom}\left(d-1, (1-p)\left(1-2\kappa^{-1}\right)\right) \leq d-\theta+1\right) \leq \frac{2}{3d}.\end{equation}
If $d\ge 2$ and $q = \frac{1}{2}d^{-4}$ and $\kappa \ge 12d^3$, then
\begin{equation} \label{binomial small 3}
P\left(\textnormal{Binom}\left(d, (1-q)\left(1-2\kappa^{-1}\right)\right) \leq d-2 \right) \le q,
\end{equation}
and
\begin{equation}\label{binomial small 4}\mathbb P \left(\textnormal{Binom}\left(d-1, (1-q)\left(1-2\kappa^{-1}\right)\right) \leq d-2\right) \leq \frac{1}{3d^2}.\end{equation}
\end{lemma}

\begin{proof} Rewriting the left side of~\eqref{binomial small 1} in terms of the complementary binomial distribution then applying Lemma~\ref{binomial bound}, we have

\begin{equation}\label{binsm1}
\begin{aligned}
\mathbb P\left(\textnormal{Binom}\left(d, (1-p)\left(1-2\kappa^{-1}\right)\right) \leq d-\theta \right) & \\
&\hspace{-3cm} = \mathbb P\left(\textnormal{Binom}\left(d, 1-(1-p)\left(1-2\kappa^{-1}\right)\right) \ge \theta \right) \\
&\hspace{-3cm} \le \mathbb P\left(\textnormal{Binom}\left(d, p+2\kappa^{-1}\right) \ge \theta \right)\\
&\hspace{-3cm} \le \left(\frac{d(p+2\kappa^{-1}) e}{\theta}\right)^{\theta}\\
&\hspace{-3cm} = \left(\frac{edp}{\theta}+\frac{2ed}{\kappa\theta} \right)^{\theta}.
\end{aligned}
\end{equation}
We verify that this is less than $p$ by showing each summand within the parentheses is smaller than $\frac12 p^{1/\theta}$. Since $\theta\ge 3$, we have
$$
\frac{edp}{\theta} = \frac{1}{3\theta d^{1/(\theta-2)}} \le \frac{1}{2}\cdot \frac{1}{(3e)^{1/\theta}} \cdot\frac{1}{d^{(\theta-1)/[\theta(\theta-2)]}} = \frac12 p^{1/\theta},
$$
which bounds the first summand. For the second summand, we have
$$
\frac{2ed}{\kappa\theta} \le \frac{2ed}{\kappa(\theta-2)}\le \frac{2ed}{9e d^{1 + \frac1{\theta-2}}} \le \frac{1}{2}\cdot \frac{1}{(3e)^{1/\theta}} \cdot\frac{1}{d^{(\theta-1)/[\theta(\theta-2)]}} = \frac12 p^{1/\theta},
$$
which completes the verification that the right side of~\eqref{binsm1} is at most $p$, and
 the proof of~\eqref{binomial small 1}. The proof of~\eqref{binomial small 3} is similar.

To prove~\eqref{binomial small 2}, we proceed in the same way as in~\eqref{binsm1} to obtain
\begin{equation}\label{binsm2}
\mathbb P \left(\textnormal{Binom}\left(d-1, (1-p)\left(1-2\kappa^{-1}\right)\right) \leq d-\theta+1\right) \leq \left(\frac{edp}{\theta-2}+\frac{2ed}{\kappa(\theta-2)} \right)^{\theta-2},
\end{equation}
where we have bounded $d-1$ above by $d$. We will bound from above each summand within the parentheses on the right side of~\eqref{binsm2} by $\frac13 d^{-1/(\theta-2)}$, which yields~\eqref{binomial small 2}, since $(2/3)^{\theta-2}\le 2/3$ for $\theta\ge 3$. For the first summand, we have
$$
\frac{edp}{\theta-2}  =  \frac{1}{3(\theta-2)} d^{-1/(\theta-2)} \le \frac13 d^{-1/(\theta-2)},
$$
and for the second summand we have
$$
\frac{2ed}{\kappa(\theta-2)} \le \frac{2ed}{9e d^{1 + \frac1{\theta-2}}} < \frac{1}{3} d^{-1/(\theta-2)},
$$
which completes the proof of~\eqref{binomial small 2}. 

The inequality~\eqref{binomial small 4} follows from a simple first-moment calculation.
\end{proof}

The analogues to Lemmas~\ref{rigid subtree} and~\ref{path rigid} are stated below, and we omit their proofs in the case $d\ge\theta\ge 3$ because they are nearly identical to their counterparts, but with Lemma~\ref{binomial small theta} used in place of Lemma~\ref{cdf bounds}. For $d\ge\theta=2$ we sketch the required modification. Recall that $\mathcal{T}_d$ denotes the infinite, full $d$-ary tree, and $\bar \xi$ is a uniform random coloring of the vertices in $\mathcal{T}_d$.
\begin{lemma}\label{strongly rigid subtree}
Let  $E$ be the event that $\bar\xi$ does not contain a strongly rigid $(\theta -1)$-fort that includes the root of $\mathcal{T}_d$. If $d\ge\theta\ge 3$ and $p = \frac{1}{3e} d^{-(\theta-1)/(\theta-2)}$ and $\kappa(\theta-2)\ge 9e d^{1 + \frac1{\theta-2}}$, then   $\mathbb P( E) \leq p$. If $d\ge \theta=2$ and $q= \frac12 d^{-4}$ and $\kappa\ge 12d^3$, then $\mathbb P(E)\le q$.
\end{lemma}

\begin{lemma}\label{path strongly rigid}
If $d\ge\theta\ge 3$ and $\kappa(\theta-2)\ge 9e d^{1 + \frac1{\theta-2}}$ or if $d\ge \theta=2$ and $\kappa\ge 12d^3$, then with probability 1 every infinite non-backtracking path in $T_d$ starting at $\rho$ will intersect a strongly rigid $(\theta-1)$-fort. 
\end{lemma}
\begin{proof}
The proof for $d\ge\theta\ge 3$ is similar to that of Lemma~\ref{path rigid}. When $d\ge \theta=2$, we must modify the the proof of Lemma~\ref{path rigid} as follows. 
Let $E$ be the event in Lemma~\ref{strongly rigid subtree}. In an infinite rooted tree with $d-1$ children at the root and $d$ children everywhere else, we let $E'$ be the event that there does not exist a strongly rigid $(d-1)$-ary subtree $S$ that contains the root. Then
\begin{equation*}
\mathbb P (E') =\mathbb P\left(\text{Binom}\left(d-1, (1-\mathbb P(E))\left(1-2/\kappa\right)\right) \le d-2\right) \leq \frac{1}{3d^2}
\end{equation*}
by Lemmas~\ref{binomial small theta} and~\ref{strongly rigid subtree}.

Let $(v_\ell)_0^{n}$ be a non-backtracking path in $T_d$ of length $n$ such that $v_0=\rho$, and define the events 
$$
F_k := \{v_{2k-1}\text{ and } v_{2k}  \text{ are in a strongly rigid $1$-fort disjoint from } \{v_\ell: \ell\notin \{2k-1, 2k\} \}\}
$$
 for $k=1, \ldots, \lfloor n/2 \rfloor$. For fixed $n$, note that $F_1, \ldots, F_{\lfloor n/2 \rfloor}$ are independent. Additionally, $\mathbb P(F_k) \geq(1-\frac{2}{\kappa})\left[1 - \mathbb P(E')\right]^2 \ge 1-\frac{2}{\kappa} - 2P(E') $ for $k=1, \ldots, \lfloor n/2 \rfloor$, since $F_k$ occurs if $v_{2k-1}$ and $v_{2k}$ are the roots of strongly rigid $(d-1)$-ary subtrees within the forest obtained by removing the edges along the path $(v_\ell)$ and the edge $\{v_{2k-1}, v_{2k}\}$ is strongly rigid.  So we have
\begin{align*}
\mathbb P((v_k)_0^n &\text{ does not intersect a strongly rigid $1$-fort})\\
&\leq \mathbb P \left(\bigcap_{k=1}^{\lfloor n/2 \rfloor} F_k^c\right) =\prod_{k=1}^{\lfloor n/2 \rfloor}  \left[1 - \mathbb P(F_k) \right] \leq  \left (\frac{2}{\kappa} + 2\mathbb P(E')\right)^{\lfloor n/2 \rfloor} \leq  \left(\frac{5}{6d^{2}}\right)^{(n-1)/2}.
\end{align*}
Let $E_n$ be the event that at least one such path of length $n$ does not intersect a rigid $1$-fort. Then
$$\mathbb P(E_n) \leq (d+1)d^n \cdot \left(\frac{5}{6d^{2}}\right)^{(n-1)/2} \leq (d+1)d \left(\sqrt{\frac{5}{6}}\right)^{n-1}$$
and so 
$$ \sum_{n=1}^\infty \mathbb P(E_n)  \leq \sum_{n=1}^\infty(d+1)d \left(\sqrt{\frac{5}{6}}\right)^{n-1}< \infty.$$
Thus by the Borel-Cantelli lemma, all infinite non-backtracking paths will intersect a strongly rigid $1$-fort almost surely. 
\end{proof}

\begin{proof}[Proof of Theorem~\ref{thm:CCA small theta}]
The proof is identical to the proof of Theorem~\ref{thm:fixation}, but with Lemmas~\ref{lem:strongly rigid k-fort},~\ref{strongly rigid subtree} and~\ref{path strongly rigid} replacing Lemmas~\ref{lem:rigid k-fort},~\ref{rigid subtree} and~\ref{path rigid}.
\end{proof}

\subsection*{Fixation for GHM}
The method of the last section also yields fixation for GHM in the specified region of parameter space, but we now give a simple argument that yields a larger fixation region, at least for fixed $\theta$ and larger $d$. The first step is the following lemma, which gives a necessary condition for the root $\rho\in T_d$ to be excited at time $t$.

\begin{lemma}\label{lem:GHM excited}
Let $d\ge \theta\ge 2$ and $\kappa\ge 3$. For every $t\ge 1$, if $\gamma_t(\rho)=1$, then there exists a $\theta$-ary subtree $S\subseteq T_d$ rooted at $\rho$ such that
\begin{itemize}
\item $\gamma_0(v) = 1$ for all $v\in S$ at distance $t$ from $\rho$,
\item $\gamma_0(v) = 0$ for all $v\in S$ at distance $t-1$ from $\rho$, and
\item for $m\in \{1, \ldots, \kappa-2\}$, we have $\gamma_0(v) \in \{0,\kappa-1, \ldots, \kappa-m\}$ for all $v\in S$ at distance $t-1-m$ from $\rho$,
\end{itemize}
where the set of vertices at a negative distance from $\rho$ is empty.
\end{lemma}

\begin{proof}

Given $\gamma_0$ and $t\ge 1$ such that $\gamma_t(\rho)=1$, we will construct $S$ up to level $t$, and confirm that it satisfies the following stronger conditions. For each $0\le s\le t$ we have
\begin{itemize}
\item $\gamma_{t-s}(v) = 1$ for all $v\in S$ at distance $s$ from $\rho$,
\item $\gamma_{t-s}(v) = 0$ for all $v\in S$ at distance $s-1$ from $\rho$, and
\item for $m\in \{1, \ldots, \kappa-2\}$, we have $\gamma_{t-s}(v) \in \{0,\kappa-1,\kappa-2, \ldots, \kappa-m\}$ for all $v\in S$ at distance $s-1-m$ from $\rho$.
\end{itemize}

We start by including $\rho$ in $S$, and note that $S$ will satisfy the conditions above for $s=0$, which reduce to $\gamma_t(\rho)=1$.  For $s=1$, since $\gamma_t(\rho)=1$, it must be that $\rho$ has at least $\theta$ neighbors that are in state $1$ at time $t-1$, and $\rho$ must be in state $0$ at time $t-1$. We therefore include in $S$ the first $\theta$ of these neighbors according to some (arbitrary) ordering on the vertices of $T_d$. Also, this confirms that $S$ satisfies the conditions above for $s=1$ and $s=0$. Now we argue inductively, assuming we have constructed $S$ up to level $1\le r <t$, and the conditions above are satisfied for all $s\le r$. For each $v\in S$ at distance $r$ from $\rho$, we have $\gamma_{t-r}(v)=1$ and its parent $w\in S$, which is at distance $r-1\ge 0$ from $\rho$ has $\gamma_{t-r}=0$. Therefore, $v$ must have at least $\theta$ children in $T_d$ that are in state $1$ at time $t-(r+1)$, and we include the first $\theta$ of these in $S$ for each such $v$. Moreover, we must have $\gamma_{t-(r+1)}(v)=0$. This verifies $S$ satisfies the first two conditions for $s=r+1$. For the third condition, note that for each $u\in S$ at distance $r-1-m \ge 0$ from $\rho$, we have $\gamma_{t - (r-m)}(u) = 0$ by the induction hypothesis. If it were the case that $\gamma_{t-r}(u) \in \{1, 2, \ldots, \kappa-m-1\}$, then we would have $\gamma_
{t-(r-m)}(u) \in \{1+m, 2+m, \ldots, \kappa-1\}$, and in particular, $\gamma_{t-(r-m)}(u) \neq 0$, a contradiction. This confirms the third condition. The construction of $S$ beyond level $t$ can be done arbitrarily, so this completes the proof.
\end{proof}

We next need a bound on the number of $\theta$-ary subtrees of depth $t$ in $T_d$.
\begin{lemma}\label{lem:subtree bound}
Let $d\ge \theta\ge 2$. The number of full $\theta$-ary subtrees of depth $t\ge 1$ that are rooted at $\rho$ in $T_d$ is
$$
{d+1 \choose \theta}  {d\choose \theta}^{\frac{\theta^t - \theta}{\theta-1}}.
$$
\end{lemma}
\begin{proof}
Let $M_t$ denote the number of full $\theta$-ary subtrees of depth $t\ge 1$ rooted at $\rho$. There are ${d+1 \choose \theta}$ ways to choose the children of $\rho$ in a $\theta$-ary subtree, so  $M_1= {d+1 \choose \theta}$. For each non-leaf, non-root vertex in the subtree, there are ${d\choose \theta}$ ways to choose its children, and there are $\theta^s$ vertices in the subtree at depth $s$. Therefore, for $s\ge 1$,
$$
M_{s+1} = M_s {d\choose \theta}^{\theta^s}.
$$
It follows that
$$
M_t = {d+1 \choose \theta}  {d\choose \theta}^{\theta + \theta^2 + \cdots + \theta^{t-1}},
$$
which is equal to the claimed formula.
\end{proof}

\begin{proof}[Proof of Theorem~\ref{thm:GHM small theta}]
For $t\ge \kappa$, by the union bound and Lemmas~\ref{lem:GHM excited} and~\ref{lem:subtree bound}, we have
\begin{equation}\label{eq:GHM excited bound}
\begin{aligned}
\mathbb{P}(\gamma_t(\rho)=1) &\le {d+1 \choose \theta}  {d\choose \theta}^{\theta^t/(\theta-1)} \times \left(\frac1\kappa\right)^{\theta^t+\theta^{t-1}}  \left(\frac2\kappa\right)^{\theta^{t-2}}  \left(\frac3\kappa\right)^{\theta^{t-3}} \cdots  \left(\frac{\kappa-1}\kappa\right)^{\theta^{t-\kappa+1}} \\
& \le {d+1 \choose \theta} \left[ \left(\frac{de}{\theta}\right)^{\theta^\kappa/(\theta-1)} \frac{2^{\theta^{\kappa-3}} 3^{\theta^{\kappa-4}} \cdots (\kappa-1)  } { \kappa^{(\theta^\kappa-1)/(\theta-1)} } \right]^{\theta^{t-\kappa+1}},
\end{aligned}
\end{equation}
since ${d\choose \theta}\le (de/\theta)^\theta$. We bound the product in the numerator by
\begin{equation}\label{product bound}
\begin{aligned}
\left[2^{\theta^{\kappa-3}} 3^{\theta^{\kappa-4}} \cdots (\kappa-1)\right]^{(\theta-1)/(\theta^{\kappa}-1)} &\le \left[2^{\theta^{\kappa-3}} 3^{\theta^{\kappa-4}} \cdots (\kappa-1)\right]^{1/\theta^{\kappa-1}}\\
&\le \prod_{k=2}^\infty k^{1/\theta^{k}}\\
& = \exp\left[ \sum_{k=2}^\infty \frac1{\theta^k} \ln(k)\right]\\
&\le  \exp\left[ \sum_{k=2}^\infty \frac1{\theta^k} \frac{k}{2} \right]\\
&= \exp\left[\frac{2\theta  - 1}{2\theta(\theta-1)^2}\right]\\
& \le e^{3/4}.
\end{aligned}
\end{equation}
Therefore, from~\eqref{eq:GHM excited bound}, we have that for $t\ge \kappa$,
$$
\mathbb{P}(\gamma_t(\rho)=1) \le {d+1 \choose \theta} \left[ \left(\frac{de}{\theta}\right)^{\theta^\kappa/(\theta-1)} \left(\frac{ e^{3/4}} { \kappa }\right)^{(\theta^\kappa-1)/(\theta-1)} \right]^{\theta^{t-\kappa+1}} \le {d+1 \choose \theta} e^{-\theta^{t-\kappa+1}},
$$
whenever $\kappa\ge e\left(\frac{de}{\theta}\right)^{1+ 1/(\theta^\kappa - 1)}$ and $\theta\ge 2$ and $\kappa\ge 3$. It follows that for $t\ge \kappa$, 
$$
\begin{aligned}
\mathbb{P}(\gamma_s(\rho)=1 \text{ for some $s\ge t$}) &\le {d+1 \choose \theta} \sum_{s\ge t} e^{-\theta^{s-\kappa+1}} \\
&\le {d+1 \choose \theta} e^{-\theta^{t-\kappa+1}}
 \sum_{s\ge 0} e^{-s} \\
 &\le 2{d+1 \choose \theta} e^{-\theta^{t-\kappa+1}}.
 \end{aligned}
$$
This gives the claimed tail bound on $\tau$, and implies fixation.

When $\theta=2, 3$, we use the bounds ${d\choose 2} \le d^2/2$ or ${d\choose 3}\le d^3/6$ in~\eqref{eq:GHM excited bound} (instead of ${d\choose \theta}\le (de/\theta)^\theta$), and repeat the same argument to get fixation for $\kappa > e^{3/4} \left(\frac{d}{\sqrt{2}}\right)^{1+1/(2^{\kappa}-1)}$ when $\theta=2$, and for $\kappa > e^{5/24} \left(\frac{d}{\sqrt[3]{6}}\right)^{1+1/(3^{\kappa}-1)}$ when $\theta=3$. In the case $\theta=3$, we used the penultimate line of~\eqref{product bound} to obtain the exponent $\frac{2\theta-1}{2\theta(\theta-1)^2} = \frac5{24}$. It is now straightforward to verify the bounds on $\kappa$ which guarantee fixation for $d$ given in the table. Moreover, when $d=\theta$, we have ${d\choose \theta}=1$, and it is easy to check from the first inequality in~\eqref{eq:GHM excited bound} that fixation occurs for all $\kappa$.  
\end{proof}

 \section{Open questions}
We suggest a few directions for further research into CCA and GHM on trees. The first natural question is what happens for the small values of $\kappa, d$ and $\theta$ not covered by our theorems. 
\begin{question}
Does fluctuation or fixation occur for GHM (or CCA) when $d=3$, $\theta=2$ and $\kappa=3, 4$? What about $d=4$, $\theta=2$ and $\kappa\le 6$?
\end{question}
A second question is whether there are other possibilities besides fluctuation or fixation.
\begin{question}
For GHM or CCA, do there exist values of $d\ge \theta\ge 2$ and $\kappa\ge 3$ such that 
$$
\mathbb{P}(\text{the state of } \rho \text{ changes infinitely often})\in (0,1)?
$$
\end{question}
In our proof of fluctuation, we in fact show that every vertex is eventually periodic with period $\kappa$. However, when $\theta=1$ and $\kappa=3$, \cite{gravner2018} show that the asymptotic rate of fluctuation is less than one for many infinite trees.
\begin{question}
Do there exist $d\ge\theta\ge 2$ and $\kappa\ge 3$ such that fluctuation occurs almost surely, but with positive probability $\rho$ is not eventually $\kappa$-periodic?
\end{question}
Finally, one could investigate all of these questions for non-uniform initial coloring distributions or on non-regular trees.


\bibliography{AutomataOnTd}
\bibliographystyle{plain}
 
\end{document}